\documentclass[12pt,reqno]{amsart}
\usepackage{graphicx}
\usepackage{verbatim}
\usepackage{textcomp}
\usepackage{amssymb}
\usepackage{cite}
\usepackage{amsmath}
\usepackage{latexsym}
\usepackage{amscd}
\usepackage{amsthm}
\usepackage{mathrsfs}
\usepackage{xypic}
\usepackage{bm}
\usepackage{url}
\usepackage{hyperref}
\usepackage{amsfonts}
\usepackage{times,graphicx,hyperref,mathrsfs}
\usepackage{CJK}
\usepackage{color}
\usepackage{etoolbox,xstring,mfirstuc,textcase}
\newtheorem{thm}{Theorem}[section]
\newtheorem{corr}[thm]{Corollary}
\newtheorem{lem}[thm]{Lemma}
\newtheorem{prop}[thm]{Proposition}

\theoremstyle{definition}

\theoremstyle{remark}
\newtheorem{rem}{\bf Remark}[section]
\numberwithin{equation}{section}
\def\supp{{\rm{\,supp\,}}}

\allowdisplaybreaks

\begin{document}

\title{\small Sharp function and weighted $L^{p}$ estimates for pseudo-differential operators with symbols in general H\"{o}rmander classes}
%{\uppercase{Sharp function and weighted $L^{p}$ estimates for pseudo-differential operators with symbols in general H\"{o}rmander classes}}
%\titlerunning{pseudo-differential operators with symbols in general H\"{o}rmander classes}
\author{Guangqing Wang}
%\author{Wenyi Chen$^{1,2}$}
\address{ School of Mathematics and Statistics, Fuyang Normal University, Fuyang, Anhui 236041, P.R.China}
%\address{$2.$  College of Mathematics and System Science, Xinjiang University, Xinjiang 830046, P.R.China}

\email{wanggqmath@whu.edu.cn(G.Wang)}
%\email{wychencn@whu.edu.cn (W.Chen)}

\pagestyle{plain}

\maketitle

\begin{abstract}
 The purpose of this paper is to prove pointwise inequalities and to establish  the boundedness on weighted $L^{p}$ spaces for pseudo-differential operators $T_{a}$ defined by the symbol $a\in S^{m}_{\varrho,\delta}$ with $0\leq\varrho\leq1,$ $0\leq\delta<1$.

Firstly, we prove that if $m\leq-n(1-\varrho)/2$, then
$$(T_{a}u)^{\sharp}(x)\lesssim M(|u|^{2})^{1/2}(x)$$
for all $x\in\mathbb{R}^{n}$ and all Schwartz function $u$. Secondly, it is shown that if $1\leq r\leq2$ and $m\leq-\frac{n}{r}(1-\varrho)$, then for any $\omega$ belongs to the class of Muckenhoupt weights $A_{p/r}$ with $r<p<\infty$, these operators are bounded on $L^{p}_\omega$. Moreover, these results are sharp on the bound of $m$.
\end{abstract}



\section{Introduction and main results}

As we all know, the theory of pseudo-differential operators initiated by Kohn and Nirenberg \cite{Nirenberg} and H\"{o}rmander \cite{Hormander} have played an important role in the analysis of linear partial differential equations (PDEs). These operators can be written in the form
\begin{eqnarray}\label{df}
T_{a}u(x)
&=&\frac{1}{(2\pi)^n}\int_{\mathbb{R}^n} e^{ i \langle x,\xi\rangle}a(x,\xi)  \hat{u}(\xi)d\xi,
\end{eqnarray}
where the symbol $a$ is a amplitude function. The most widely used class of amplitudes are those, introduced by H\"{o}rmander in \cite{Hormander2}, named $S^{m}_{\varrho,\delta}$ class that consists of $a\in C^{\infty}(\mathbb{R}^{n}\times\mathbb{R}^{n})$ with
\begin{eqnarray*}
|\partial^{\beta}_x\partial^{\alpha}_{\xi}a(x,\xi)|\leq C_{\alpha,\beta}\langle\xi\rangle^{m-\varrho |\alpha|+\delta|\beta|},
\end{eqnarray*}
for $m\in\mathbb{R},$ $\varrho,\delta\in[0,1]$ and any~multi-indices $\alpha,\beta.$ In order to state the backgrounds of relevant research and our results we begin by introducing some notations used in this paper. For a locally integrable function $u$, the sharp function $u^{\sharp}$ and Hardy-Littlewood maximal function $Mu(x)$ are defined by the formula
$$u^{\sharp}(x)=\sup\limits_{x\in Q}\inf\limits_{c}\frac{1}{|Q|}\int_{Q}|u(y)-c|dy\quad {\rm and}\quad Mu(x)=\sup\limits_{x\in Q}\frac{1}{|Q|}\int_{Q}|u(y)|dy$$
where the supremum is taken over all cubes $Q$ containing $x$ and its sides parallel to the coordinate axes, and $c$ takes values all complex number. And for any $1\leq p<\infty$, the generalized  Hardy-Littlewood maximal function is defined by the function $M(|u|^{p})^{1/p}(x)$. See \cite {Fefferman1} for more details.

One of the most important topics about pseudo-differential operators is whether they are bounded on weighted Lebesgue space, which follows partly from the direct pointwise estimates for these operators. So, there are many researchers pay close attention to these estimates, such as Miller\cite{Miller}, Chanillo\cite{Chanillo}, Yabuta\cite{Yabuta}, Miyachi\cite{Miyachi} and so on. For more researches, we refer to \cite{Journe,Nishigaki,Tang,Hounie,Wang}. Here we would like to mention some results that are closely related to the problem we focus on.  For the symbol $a(x,\xi)\in S^{0}_{1,\delta}$ with $0\leq\delta<1$, Journ\'{e}\cite{Journe} proved that for each $1<p<\infty$
$$(T_{a}u)^{\sharp}(x)\lesssim M_{p}u(x)$$
for all $x\in\mathbb{R}^{n}$ and $u\in C^{\infty}_{0}(\mathbb{R}^{n}),$ which would be reduced to Miller's result \cite{Miller} when $\delta=0$. And for the symbol $a(x,\xi)\in S^{-n(1-\varrho)/2}_{\varrho,\delta}$ with $0\leq\delta<\varrho<1$, Chanillo and Torchinsky \cite{Chanillo} showed
$$(T_{a}u)^{\sharp}(x)\lesssim M_{2}u(x),$$
which was extended to the case $0<\delta=\varrho<1$ by Miyachi and Yabuta \cite{Miyachi}. As for the case $0<\varrho<\delta<1$, \'{A}lvarez and Hounie \cite{Hounie} got a similar result but put a stronger condition on the parameter $m$, that is,
$$m\leq-n(1-\varrho)-\mu,$$
where $$2\mu=1+n(\varrho+\max\{0,\frac{\delta-\varrho}{2}\})- \sqrt{\big(1+n(\varrho+\max\{0,\frac{\delta-\varrho}{2}\})\big)^{2}-4n\max\{0,\frac{\delta-\varrho}{2}\}},$$ which was relaxed to
$$m\leq\frac{n(2\varrho-\delta-1)}{2}$$
by Kim and Shin \cite{Kim} in 1991.

In this paper, the bound on $m$ is extended to
$m\leq-\frac{n}{2}(1-\varrho)$ in general case.

\begin{thm}\label{T1}
Let $0\leq\varrho\leq1,$ $0\leq\delta<1$ and $a(x,\xi)\in S^{-\frac{n}{2}(1-\varrho)}_{\varrho,\delta}$. Then there is a constant $C$ independent of $a$ and $u$, such that
\begin{equation}\label{sharp}
(T_{a}u)^{\sharp}\leq CM_{2}u(x).
\end{equation}
%$$\|T_{a}u\|_{L^{1}(\mathbb{R}^{n})}\leq C\|u\|_{H^{1}(\mathbb{R}^{n})}$$
\end{thm}
Let $\omega\in L^{1}_{loc}$ be a nonnegative function. One says that $\omega$ belongs to the class of Muckenhoupt $A_{p}$ weights for $1<p<\infty$ if there exists a constant $C>0$ such that
\begin{equation}\label{Ap}
\sup\limits_{Q\subset\mathbb{R}^{n}}\big(\frac{1}{|Q|}\int_{Q}\omega(x)dx\big)
\big(\frac{1}{|Q|}\int_{Q}\omega(x)^{\frac{1}{1-p}}dx\big)^{p-1}\leq C.
\end{equation}
One says that $\omega\in A_{1}$ if there exists a constant $C>0$ such that
\begin{equation}\label{A1}
M\omega(x)\leq C\omega(x)
\end{equation}
for almost all $x\in\mathbb{R}^{n}.$ The smallest constant appearing in (\ref{Ap}) or (\ref{A1}) is called the $A_{p}$ constant of $\omega$ which is denoted by $[\omega]_{p}$. The usual notation that
$$\|u\|_{L^{p}_{\omega}}=\big(\int_{\mathbb{R}^{n}}|u(x)|^{p}\omega(x)dx\big)^{\frac{1}{p}}$$
will be adopted in this paper. In 1991, Chen \cite{Chen} got the weighted boundedness of pseudo-differential operators in the case $0\leq\delta<\varrho<1$, which is extended to the case $0\leq\varrho\leq1,$ $0\leq\delta<1$ in this paper.
\begin{thm}\label{weight}
Let $0\leq\varrho\leq1,$ $0\leq\delta<1$, $1\leq r\leq2$ and $a(x,\xi)\in S^{-\frac{n}{r}(1-\varrho)}_{\varrho,\delta}$. Suppose $\omega\in A_{p/r}$ with $r<p<\infty$. Then there is a constant $C$ independent of $a$ and $u$, such that
$$\|T_{a}u\|_{L^{p}_{\omega}}\leq C\|u\|_{L^{p}_{\omega}}$$
\end{thm}

\begin{rem}
For $0<\delta<\varrho\leq1$ Chanillo and Torchinsky  \cite{Chanillo} raise a question: where 2 in \eqref{sharp} is the smallest index that may be used in the right hand side of the inequality. Chen \cite{Chen} gave a positive answer by the following result, which implies both Theorem \ref{T1} and Theorem \ref{weight} is sharp as well.
\end{rem}

\begin{thm}[Chen \cite{Chen}]\label{Chen}
Let $0\leq\varrho<1,$ $1\leq r\leq2$ and $1<p_{0}<\infty$. If $\forall a(\xi)\in S^{-\frac{n}{r}(1-\varrho)}_{\varrho,0} $ and $\forall \omega\in A_{q}$,
$$\|T_{a}u\|_{L^{p_{0}}_{\omega}}\leq C\|u\|_{L^{p_{0}}_{\omega}}.$$
Then $$q\leq \frac{p_{0}}{r}$$
\end{thm}

In the next section we will give some estimates for the kernel of pseudo-differential operator, and based on these estimates, the proof of Theorem \ref{T1} will be shown. In section 3, the proof of Theorem \ref{weight} will be shown.

\section{The proof of Theorem \ref{T1}}

First we introduce the standard Littlewood-Paley partition of unity. Let $C>1$ be a constant. Set
$E_{-1}=\{\xi:|\xi|\leq 2C\}$, $E_{j}=\{\xi:C^{-1}2^{j}\geq|\xi|\leq C2^{j+1}\}$, $j=0,1,2,\cdots$

\begin{lem}\label{L1}
There exist $\psi_{-1}(\xi),\psi(\xi)\in C^{\infty}_{0}$, such that
\begin{enumerate}
  \item $\supp\psi\subset E_{0}$, $\supp\psi_{-1}\subset E_{-1};$
  \item $0\leq\psi\leq1$, $0\leq\psi_{-1}\leq1;$
  \item $\psi_{-1}(\xi)+\sum\limits^{\infty}_{j=1}\psi(2^{-j}\xi)=1.$
\end{enumerate}
\end{lem}

%Let
%
%\begin{equation}\label{K}
%K(x,w)=\int_{\mathbb{R}^n} e^{ i  \langle w,\xi\rangle}a(x,\xi)  d\xi.
%\end{equation}
%Then we can write
%\begin{eqnarray}\label{df}
%T_{a}u(x)
%=\int_{\mathbb{R}^n}K(x,x-y)u(y)dy
%\end{eqnarray}
By Lemma \ref{L1}, the operator $T_{a}$ can be decomposed as
\begin{eqnarray}\label{de}
T_{a}u(x)=\sum\limits_{j=0}^{\infty}T_{j}u(x)
\end{eqnarray}
where
\begin{eqnarray*}
T_{j}u(x)
=\int_{\mathbb{R}^n}K_{j}(x,x-y)u(y)dy
\end{eqnarray*}
with
$$K_{j}(x,w)=\int_{\mathbb{R}^n} e^{ i  \langle w,\xi\rangle}a(x,\xi)\psi(2^{-j}\xi)  d\xi.$$
Now we will make some necessary estimates for the operators $T_{j}$ and its kernel $K_{j}(x,y)$.

\begin{lem}\label{L3}
Let $Q(x_{0},l)$ be a fixed cube with side length $l<1$. Suppose $0\leq\varrho\leq1,$ $0\leq\delta<1$ and $a(x,\xi)\in S^{-\frac{n}{2}(1-\varrho)}_{\varrho,\delta}$. Then for any positive integer $j$ satisfying $2^{j}l<1$
$$\int_{\mathbb{R}^{n}}|u(y)||K_{j}(x,x-y)-K_{j}(z,z-y)|dy\leq C 2^{j}lM_{2}u(x_{0}),\quad \forall x,z\in Q(x_{0},l)$$
where $C$ is a constant independent of $a$ and $\varphi$.
\end{lem}

\begin{proof}
The idea behind the proof is standard which could be found in \cite{Chanillo}. For convenience, we list the detail here. First, integrand on the left above can be bounded by
\begin{eqnarray}\label{E12}
&&\int_{\mathbb{R}^{n}}|u(y)||\int_{\mathbb{R}^n}e^{  i\langle x-y,\xi\rangle}\psi(2^{\nu}l\xi)\big(a(x,\xi)-a(z,\xi)\big) d\xi|dy\\
&&+\int_{\mathbb{R}^{n}}|u(y)||\int_{\mathbb{R}^n}e^{i\langle x-y,\xi\rangle}\psi(2^{\nu}l\xi)a(z,\xi)\big(e^{ i \langle x-z,\xi\rangle}-1\big) d\xi|dy
=:I_{1}+I_{2}.\nonumber
\end{eqnarray}

We consider $I_{1}$ firstly. Break up  this integrand as follows
\begin{eqnarray*}
&&\int_{|y-x_{0}|\leq2^{-j\varrho+1}}
+
\int_{|y-x_{0}|>2^{-j\varrho+1}}
=:I_{11}+I_{12}
\end{eqnarray*}

For the first part $I_{11}$, H\"{o}lder's inequality and Parseval's identity show that it is bounded by
\begin{eqnarray}\label{b2}
&&|x-z|\big(\int_{|y-x_{0}|\leq2^{-j\varrho+1}}|u(y)|dy\big)^{\frac{1}{2}}\nonumber\\
&&\quad\times\big(\int_{|y-x_{0}|\leq2^{-j\varrho+1}}|\int_{\mathbb{R}^n} e^{ i \langle x-y,\xi\rangle}\big(\nabla_{x}a\big) ( \tilde{x},\xi)\psi(2^{-j}\xi) d\xi|^{2}dy\big)^{\frac{1}{2}}\nonumber\\
&\lesssim&l2^{-\frac{jn\varrho}{2}}\big(\int_{\mathbb{R}^n} |\big(\nabla_{x}a\big) ( \tilde{x},\xi)\psi(2^{-j}\xi)|^{2} d\xi\big)^{\frac{1}{2}}M_{2}u(x_{0})\nonumber\\
&\lesssim&2^{j\delta}lM_{2}u(x_{0})
\end{eqnarray}

To estimate $I_{12}$, we use H\"{o}lder's inequality, integrating by parts, Parseval's identity and the fact  $|y-x_{0}|\sim|y-x|$ that follows from $2^{j}l<1$, $x\in Q(x_{0},l)$ and $|y-x_{0}|>2^{-j\varrho+1}$. They gives that $I_{12}$ is bounded by
\begin{eqnarray}\label{b3}
&&|x-z|\big(\int_{|y-x_{0}|>2^{-j\varrho+1}}\frac{|u(y)|}{|y-x_{0}|^{2N}}dy\big)^{\frac{1}{2}}\nonumber\\
&&\quad\times\big(\int_{|y-x_{0}|>2^{-j\varrho+1}}|y-x_{0}|^{2N}|\int_{\mathbb{R}^n} e^{ i \langle x-y,\xi\rangle}\big(\nabla_{x}a\big) ( \tilde{x},\xi)\psi(2^{-j}\xi) d\xi|^{2}dy\big)^{\frac{1}{2}}\nonumber\\
&\lesssim&l2^{-j\varrho(\frac{n}{2}-N)}\sum\limits_{|\alpha|=N}\big(\int_{\mathbb{R}^n} |\partial^{\alpha}_{\xi}\big(\nabla_{x}a\big) ( \tilde{x},\xi)\psi(2^{-j}\xi)|^{2} d\xi\big)^{\frac{1}{2}}M_{2}u(x_{0})\nonumber\\
&\lesssim&2^{j\delta}lM_{2}u(x_{0})
\end{eqnarray}
So, (\ref{b2}) and (\ref{b3}) imply that
$$I_{1}\lesssim 2^{j\delta}lM_{2}u(x_{0}).$$

With the same argument as above, we can get
\begin{eqnarray*}
I_{2}\lesssim 2^{j}lM_{2}u(x_{0}).
\end{eqnarray*}
Thus the desired estimate can be got immediately.
\end{proof}

\begin{lem}\label{La}
Let $Q(x_{0},l)$ be a fixed cube with side length $l<1$. Suppose $0<\varrho<\delta<1$, $a\in S^{-\frac{n}{2}(1-\varrho)}_{\varrho,\delta}$, then for any $1\leq\lambda\leq\frac{1}{\varrho}$ and any positive integer $N>\frac{n}{2}$
\begin{eqnarray*}
\frac{1}{|Q|}\int_{Q(x_{0},l)}|T_{j}u(x)|dx
&\lesssim&M_{2}u(x_{0})\big(l^{\lambda}2^{j\delta}+ 2^{-j\frac{n}{2}(1-\varrho)}l^{-\frac{\lambda n}{2}(1-\varrho)}+l^{\lambda\varrho(\frac{n}{2}-N)}2^{j\varrho(\frac{n}{2}-N)}\big).
\end{eqnarray*}
\end{lem}
\begin{proof}
If $1<\lambda\leq\frac{1}{\varrho}$. Then $l^{\lambda}<l$ since $l<1$. Take integer $L$ such that it is the first number no less than  $l^{1-\lambda}$, that is $L-1<l^{1-\lambda}\leq L$. Then there are $L^{n}$ cubes with the same side length $l^{\lambda}$ covering $Q(x_{0},l)$. Moreover, we have
$$Q(x_{0},l)\subset\cup_{i=1}^{L^{n}}Q(x_{i},l^{\lambda})\subset Q(x_{0},2l)$$
Clearly, the number of these cubes $L^{n}$ less than $2^{n}l^{n(1-\lambda)}$

Denote
\begin{eqnarray}\label{0A}
T_{j,i}u(x)
&=&\int_{\mathbb{R}^n} e^{ i \langle x,\xi\rangle}a(x_{i},\xi) \psi(2^{-j}\xi) \hat{u}(\xi)d\xi.
\end{eqnarray}
We write
\begin{eqnarray}\label{A}
&&\frac{1}{|Q|}\int_{Q(x_{0},l)}|T_{j}u(x)|dx\nonumber\\
&\leq&\frac{1}{|Q|}\sum\limits_{i=1}^{L^{n}}\bigg(\int_{Q(x_{i},l^{\lambda})}|T_{j}u(x)-T_{j,i}u(x)|dx+\int_{Q(x_{i},l^{\lambda})}|T_{j,i}u(x)|dx\bigg).
\end{eqnarray}
We claim that
\begin{eqnarray}\label{a}
|T_{j}u( x)-T_{j,i}u( x)|\lesssim M_{2}u(x_{0})|x-x_{i}|2^{j\delta}
\end{eqnarray}
\begin{eqnarray}\label{b}
\int_{Q(x_{0},l)}|T_{j,i}u(x)|dx
&\lesssim&M_{2}u(x_{0})\big(2^{jm}l^{\frac{\lambda n}{2}+\frac{\varrho\lambda n}{2}}+l^{n\lambda}l^{\lambda\varrho(\frac{n}{2}-N)}2^{j(m-\varrho N+\frac{n}{2})}\big)
\end{eqnarray}
Recall that $L^{n}\leq2^{n}l^{n(1-\lambda)}$. Substituting both of them into (\ref{A}), we can get the desired estimate.

First, we prove the estimate (\ref{a}). Note that $|T_{j}u( x)-T_{j,i}u( x)|$ can be bounded by
\begin{eqnarray*}
\int_{\mathbb{R}^n}|u(y)||\int_{\mathbb{R}^n} e^{ i \langle x-y,\xi\rangle}\big(a(x,\xi)-a(x_{i},\xi)\big)\psi(2^{-j}\xi) d\xi|dy
\end{eqnarray*}
Then, \eqref{a} follows from the same argument as \eqref{E12}.

Now, we prove (\ref{b}). For fixed $x_{i}$, we can see that $a(x_{i},\xi) \psi(2^{-j}\xi)\in S^{0}_{\varrho, 0}$ with the bounds $\lesssim 2^{-j\frac{n}{2}(1-\varrho)}$. So
$$\|T_{j,i}u\|_{L^{2}}\lesssim2^{-j\frac{n}{2}(1-\varrho)}\|u\|_{L^{2}}.$$
Set
\begin{eqnarray}\label{ui}
u_{i,1}(x)=u(x)\chi_{Q(x_{i},2l^{\varrho\lambda})}(x) \quad{\rm and} \quad u_{i,2}(x)=u(x)-u_{i,1}(x),
\end{eqnarray}
where $\chi_{Q(x_{i},2l^{\varrho\lambda})}(x)$ is the characteristic function of the ball $Q(x_{i},2l^{\varrho\lambda}).$ Then the left hand of (\ref{b}) can be bounded by
\begin{eqnarray*}
\int_{Q(x_{i},l^{\lambda})}|T_{j,i}u_{i,1}(x)|dx+
\int_{Q(x_{i},l^{\lambda})}|T_{j,i}u_{i,2}(x)|dx=:M_{1}+M_{2}
\end{eqnarray*}

Recall that $\lambda\leq\frac{1}{\varrho}$ and $Q(x_{i},l^{\lambda})\subset Q(x_{0},2l)$. We can get $x_{0}\in Q(x_{i},2l^{\lambda \varrho})$. H\"{o}lder's inequality and $L^{2}$-boundedness of $T_{j,i}$ imply that $M_{1}$ is bounded by
\begin{eqnarray}\label{b1}
l^{\frac{\lambda n}{2}}\|T_{j,i}u_{i,1}\|_{L^{2}}
\lesssim 2^{-j\frac{n}{2}(1-\varrho)}l^{\frac{\lambda n}{2}}\|u_{i,1}\|_{L^{2}}
\lesssim 2^{-j\frac{n}{2}(1-\varrho)}l^{\frac{\lambda n}{2}+\frac{\varrho\lambda n}{2}}M_{2}u(x_{0}).
\end{eqnarray}

For $M_{2}$, noticing that any $x\in Q(x_{i},l^{\lambda})$ and any $y\in Q^{C}(x_{i},2l^{\lambda\varrho})$, we have
$$|y-x|\geq\frac{|y-x_{i}|}{2}.$$
Moveover $|y-x_{0}|\sim|y-x_{i}|$ follows form  $x_{0}\in Q(x_{i},2l^{\lambda \varrho})$ and $y\in Q^{C}(x_{i},2l^{\lambda\varrho})$.
H\"{o}lder's inequality, Integrating by parts  and Parseval's identity give that $|T_{j,i}u_{i,2}(x)|$ is bounded by
\begin{eqnarray*}
&&\big(\int_{|y-x_{i}|>l^{\lambda\varrho}}\frac{|u(y)|}{|y-x_{i}|^{2N}}dy\big)^{\frac{1}{2}}
\big(\int_{|y-x_{i}|>l^{\lambda\varrho}}|y-x_{i}|^{2N}|\int_{\mathbb{R}^n} e^{ i \langle x-y,\xi\rangle}a(x_{i},\xi)\psi(2^{-j}\xi) d\xi|^{2}dy\big)^{\frac{1}{2}}\nonumber\\
&\lesssim&\big(\int_{|y-x_{0}|>l^{\lambda\varrho}}\frac{|u(y)|}{|y-x_{0}|^{2N}}dy\big)^{\frac{1}{2}}\big(\int_{\mathbb{R}^n} |\partial^{\alpha}_{\xi}a( x_{i},\xi)\psi(2^{-j}\xi)|^{2} d\xi\big)^{\frac{1}{2}}\nonumber\\
&\lesssim&l^{\lambda\varrho(\frac{n}{2}-N)}2^{j\varrho(\frac{n}{2}- N)}M_{2}u(x_{0})
\end{eqnarray*}
 So
\begin{eqnarray}\label{b20}
M_{2}=\int_{Q(x_{i},l^{\lambda})}|T_{j,i}u_{i,2}(x)|dx\lesssim l^{n\lambda}l^{\lambda\varrho(\frac{n}{2}-N)}2^{j\varrho(\frac{n}{2}- N)}M_{2}u(x_{0})
\end{eqnarray}
Thus, the desired estimate (\ref{b}) follows from (\ref{b1}) and (\ref{b20}).

If $\lambda=1$, we define
\begin{eqnarray}
T_{j,0}u(x)
&=&\int_{\mathbb{R}^n} e^{ i \langle x,\xi\rangle}a(x_{0},\xi) \psi(2^{-j}\xi) \hat{u}(\xi)d\xi.
\end{eqnarray}
Then the desired estimate can be got by the same argument as above with $T_{j,i}u$ replaced by $T_{j,0}u$.

So we complete the proof.
\end{proof}

We remark that a similar result holds for the case $\varrho=0$ .The only change in the argument that are needed are as follows: The definition of the function $u_{i,1}(x)$ given by \eqref{ui} must be modified to read
$$u_{i,1}(x)=u(x)\chi_{Q(x_{i},2l^{\frac{n\lambda}{2N}}2^{j\frac{n}{2N}})}(x),$$
where $1\leq\lambda\leq\frac{2}{1-\delta}$ and positive integer $N>\frac{n}{1-\delta }$.
\begin{lem}\label{La0}
Let $Q(x_{0},l)$ be a fixed cube with side length $l<1$. Suppose $\varrho= 0$, $0<\delta<1$, $a\in S^{-\frac{n}{2}}_{0,\delta}$, then for any $1\leq\lambda\leq\frac{2}{1-\delta}$ and any positive integer $N>\frac{n}{1-\delta }$
\begin{eqnarray*}
\frac{1}{|Q|}\int_{Q(x_{0},l)}|T_{j}u(x)|dx
&\lesssim&M_{2}u(x_{0})\big(l^{\lambda}2^{j\delta}+ 2^{-j\frac{n}{2}(1-\frac{n}{2N})}l^{-\frac{\lambda n}{2}(1-\frac{n}{2N})}\big)
\end{eqnarray*}
\end{lem}

Next we are going to prove Theorem \ref{T1}. Here we only consider the case $\varrho<\delta<1$ since the other case has been considered in \cite{Chanillo,Journe,Miyachi,Miller}. Before illustrating the proof, we recall a fundamental $L^{2}$-estimate for the pseudo-differential operators due to Hounie \cite{Hounie}.

\begin{thm}(Hounie \cite{Hounie})\label{L2}
Assume $0\leq\varrho\leq1,$ $0\leq\delta<1$ and $a\in S^{\min(0,\frac{n}{2}(\varrho-\delta))}_{\varrho,\delta}$. Then the operator $T_{a}$ is bounded on $L^{2}.$
\end{thm}

\begin{proof}[Proof of Theorem \ref{T1}]

Without loss of generality, we assume that the symbol $a(x,\xi)$ vanishes for $|\xi|\leq 1$. Let $Q=Q(x_{0},l)$ denote the cube centered at $x_{0}$ with the side length $l.$ For any fixed cube $Q$, we are going to prove that
\begin{eqnarray}\label{E0}
\frac{1}{|Q|}\int_{Q}|T_{a}u(x)-C_{Q}|dx\leq C M_{2}u(x_{0}),
\end{eqnarray}
where $C_{Q}=\frac{1}{|Q|}\int_{Q}T_{a}u(y)dy$. As usual the proof will be divided into two cases.

Case 1. $l<1.$\quad Note that the left hand of (\ref{E0}) can be controlled by
\begin{eqnarray}\label{E1}
\frac{1}{|Q|^{2}}\int_{Q}\int_{Q}|T_{a}u(x)-T_{a}u(y)|dydx.
\end{eqnarray}
We compose the operator $T_{a}$ as (\ref{de}), then estimate (\ref{E1}) by
\begin{eqnarray}
\sum\limits_{j=1}^{\infty}\frac{1}{|Q|^{2}}\int_{Q}\int_{Q}|T_{j}u(x)-T_{j}u(z)|dzdx.
\end{eqnarray}
Since $l<1$, there is a positive integer $j_{0}$ such that $2^{j_{0}}\sim l^{-1}$. Lemma (\ref{L3}) implies that
\begin{eqnarray*}
|T_{j}u(x)-T_{j}u(z)|\leq  \int_{\mathbb{R}^{n}}|u(y)||K_{j}(x,x-y)-K_{j}(z,z-y)|dy\leq C 2^{j}|x-z|M_{2}u(x_{0})
\end{eqnarray*}
So we have
\begin{eqnarray}
\sum\limits_{j=1}^{j_{0}}\frac{1}{|Q|^{2}}\int_{Q}\int_{Q}|T_{j}u(x)-T_{j}u(z)|dzdx\lesssim M_{2}u(x_{0})l\sum\limits_{j=1}^{j_{0}}2^{j}\lesssim M_{2}u(x_{0}).
\end{eqnarray}

Clearly, it remains to show
\begin{eqnarray}\label{E6}
\sum\limits_{j=j_{0}}^{\infty}\frac{1}{|Q|^{2}}\int_{Q}\int_{Q}|T_{j}u(x)-T_{j}u(z)|dzdx\lesssim M_{2}u(x_{0}).
\end{eqnarray}

For $0<\varrho<\delta<1$, we claim that there is a positive integer $j_{\gamma}$ with $2^{j_{\gamma}}\sim l^{-\frac{1}{\delta^{\gamma}}}>l^{-\frac{1}{\varrho}}$ such that
\begin{eqnarray}\label{E3}
\sum\limits_{j=j_{1}}^{j_{\gamma}}\frac{1}{|Q|}\int_{Q(x_{0},l)}|T_{j}u(x)|dx
&\leq& \tilde{C}_{\gamma}M_{2}u(x_{0}),
\end{eqnarray}
where $\gamma\geq1$ is positive integer and $\tilde{C}_{\gamma}$ is a constant independent of $l$.  Actually, for any positive integer $k$ satisfy $\frac{1}{\delta^{k}}\leq \frac{1}{\varrho}$, there are positive integer $j_{k},j_{k+1}$ satisfy $2^{j_{k}}\sim l^{-\frac{1}{\delta^{k}}},2^{j_{k+1}}\sim l^{-\frac{1}{\delta^{k+1}}}$, respectively. Take $\lambda=\frac{1}{\delta^{k}}$ in Lemma \ref{La}.
Then we have
\begin{eqnarray}\label{3}
\sum\limits_{j=j_{k}}^{j_{k+1}}\frac{1}{|Q|}\int_{Q(x_{0},l)}|T_{j}u(x)|dx
&\leq& C_{k}M_{2}u(x_{0})\sum\limits_{j=j_{k}}^{j_{k+1}}
\big(l^{\frac{1}{\delta^{k}}}2^{j\delta}+2^{jm}l^{-\frac{ n}{2\delta^{k}}(1-\varrho)}+l^{\frac{\varrho}{\delta^{k}}(\frac{n}{2}-N)}2^{j\varrho(\frac{n}{2}-N)}\big)\nonumber\\
&\leq& C_{k}M_{2}u(x_{0}).
\end{eqnarray}
Notice that there is a positive integer $\gamma$ such that $\frac{1}{\delta^{\gamma-1}}\leq\frac{1}{\varrho}<\frac{1}{\delta^{\gamma}}$ since $0<\varrho<\delta<1$. Furthermore, there are positive integer $j_{\gamma},j_{\gamma-1}$ with $2^{j_{\gamma-1}}\sim l^{-\frac{1}{\delta^{\gamma-1}}}\leq l^{-\frac{1}{\varrho}}<l^{-\frac{1}{\delta^{\gamma}}}\sim 2^{j_{\gamma}}$. So we can write
\begin{eqnarray*}
\sum\limits_{j=j_{1}}^{j_{\gamma}}\frac{1}{|Q|}\int_{Q(x_{0},l)}|T_{j}u(x)|dx
=\big(\sum\limits_{j=j_{1}}^{j_{2}}+\sum\limits_{j=j_{2}}^{j_{3}}+...\sum\limits_{j=j_{k}}^{j_{k+1}}+...\sum\limits_{j=j_{\gamma-1}}^{j_{\gamma}}\big)
\frac{1}{|Q|}\int_{Q(x_{0},l)}|T_{j}u(x)|dx.
\end{eqnarray*}
Combine this and \eqref{3}, we can get \eqref{E3} immediately.

For $0=\varrho<\delta<1$, \eqref{E3} is valid for some positive integer $j_{\gamma}$ with $2^{j_{\gamma}}\sim l^{-\frac{1}{\delta^{\gamma}}}>l^{-\frac{2}{1-\delta}}$. The only change in the argument is that we need Lemma \ref{La0}, instead of Lemma \ref{La}.

Then, \eqref{E6} and \eqref{E3} implies that it remains to show
\begin{eqnarray}\label{E06}
\sum\limits_{j=j_{\gamma}}^{\infty}\frac{1}{|Q|}\int_{Q}|T_{j}u(x)|dx\leq CM_{2}u(x_{0})
\end{eqnarray}
In order to get this estimate, fixed $N_{0}$ such that $\frac{n}{2}<N_{0}<\frac{n}{2}(1+\frac{1-\delta}{\varrho})$ (if $\varrho=0$, fixed $N_{0}>\frac{2}{1-\delta}$)  and denote
$$\Gamma=2^{\frac{j}{N_{0}}\big(\varrho(\frac{n}{2}-N_{0})+\frac{n}{2}(1-\delta)\big)}l^{\frac{n}{2N_{0}}}.$$
Set $u_{1}(x)=u(x)\chi_{Q(x_{0},2\Gamma)}(x)$ and $u_{2}(x)=u(x)-u_{1}(x)$.
Then
\begin{eqnarray}\label{Ea}
\frac{1}{|Q|}\int_{Q(x_{0},l)}|T_{j}u(x)|dx
\leq\frac{1}{|Q|}\int_{Q(x_{0},l)}|T_{j}u_{1}(x)|dx+
\frac{1}{|Q|}\int_{Q(x_{0},l)}|T_{j}u_{2}(x)|dx
\end{eqnarray}

Notice that $a(x,\xi) \psi(2^{-j}\xi)\in S^{-\frac{n}{2}(\delta-\varrho)}_{\varrho,\delta}$ with bounds $\lesssim 2^{-j\frac{n}{2}(1-\delta)}$. H\"{o}lder's inequality and the $L^{2}$-estimate of $T_{j}$ give that
\begin{eqnarray}\label{E10}
\frac{1}{|Q|}\int_{Q}|T_{j}u_{1}(x)|dx\lesssim 2^{-j\frac{n}{2}(1-\delta)}l^{-\frac{n}{2}}\|u_{1}\|_{L^2}\lesssim 2^{-j\frac{n}{2}(1-\delta)}l^{-\frac{n}{2}}\Gamma^{\frac{n}{2}}M_{2}u(x_{0})
\end{eqnarray}

Notice that $\Gamma>l$. We have $|y-x|\sim|y-x_{0}|$ for $\forall x\in Q(x_{0},l)$ and $\forall y\in Q^{C}(x_{0},2\Gamma)$. So direct computation show that
\begin{eqnarray}\label{E11}
|T_{j}u_{2}(x)|
&\leq&\int_{|y-x_{0}|\geq 2\Gamma}|K_{j}(x,x-y)||u(y)|dy
\lesssim \Gamma^{(\frac{n}{2}-N_{0})}2^{j\varrho(\frac{n}{2}-N_{0})} M_{2}u(x_{0})
\end{eqnarray}
\begin{eqnarray*}
\sum\limits_{j=j_{\gamma}}^{\infty}(2^{-j\frac{n}{2}(1-\delta)}l^{-\frac{n}{2}}\Gamma^{\frac{n}{2}}+\Gamma^{(\frac{n}{2}-N_{0})}2^{j\varrho(\frac{n}{2}-N_{0})})
&=&2l^{-\frac{n}{2}(1-\frac{n}{2N_{0}})}\sum\limits_{j=j_{\gamma}}^{\infty}
2^{-j\big(\frac{n}{2}(1-\delta)(1-\frac{n}{2N_{0}})+\frac{n\varrho}{2N_{0}}(N_{0}-\frac{n}{2})\big)}\\
&\lesssim&l^{-\frac{n}{2}(1-\frac{n}{2N_{0}})}2^{-j_{\gamma}\big(\frac{n}{2}(1-\delta)(1-\frac{n}{2N_{0}})+\frac{n\varrho}{2N_{0}}(N_{0}-\frac{n}{2})\big)}\\
&\leq&l^{-\frac{n}{2}(1-\frac{n}{2N_{0}})}l^{\frac{1}{\varrho}\big(\frac{n}{2}(1-\delta)(1-\frac{n}{2N_{0}})+\frac{n\varrho}{2N_{0}}(N_{0}-\frac{n}{2})\big)}\leq1
\end{eqnarray*}
Form \eqref{Ea}, \eqref{E10} and \eqref{E11}, it follows that the left hand of \eqref{E06} is bounded by

\begin{eqnarray}\label{E60}
&&M_{2}u(x_{0})
\sum\limits_{j=j_{\gamma}}^{\infty}(2^{-j\frac{n}{2}(1-\delta)}l^{-\frac{n}{2}}\Gamma^{\frac{n}{2}}+\Gamma^{(\frac{n}{2}-N_{0})}2^{j\varrho(\frac{n}{2}-N_{0})})\nonumber\\
&\lesssim&M_{2}u(x_{0})
l^{-\frac{n}{2}(1-\frac{n}{2N_{0}})}2^{-j_{\gamma}\big(\frac{n}{2}(1-\delta)(1-\frac{n}{2N_{0}})+\frac{n\varrho}{2N_{0}}(N_{0}-\frac{n}{2})\big)}
\leq M_{2}u(x_{0}).
\end{eqnarray}
The last inequality holds because of $2^{j_{\gamma}} >l^{-\frac{1}{\varrho}}$ and $N_{0}>\frac{n}{2}$. Thus, the proof in this case is finished.

Case 2. $l\geq1.$\quad
Set $u_{3}(x)=u(x)\chi_{{Q(x_{0},2l)}}(x)$ and $u_{4}(x)=u(x)-u_{3}(x)$. Then
$$T_{a}u(x)=T_{a}u_{3}(x)+T_{a}u_{4}(x).$$
We estimate $T_{a}u_{3}$ firstly. Note that $S^{-\frac{n}{2}(1-\varrho)}_{\varrho,\delta}\subset S^{\min(0,\frac{n}{2}(\varrho-\delta))}_{\varrho,\delta}$ for any $0\leq\varrho\leq1$ and $0\leq\delta<1$. Schwartz's inequality and Lemma \ref{L2} yield
\begin{eqnarray}\label{Eb}
\frac{1}{|Q|}\int_{Q}|T_{a}u_{3}(x)|dx\leq|Q|^{-\frac{1}{2}}\|T_{a}u_{3}\|_{L^2}\lesssim |Q|^{-\frac{1}{2}}\|u_{3}\|_{L^2}\lesssim M_{2}u(x_{0})
\end{eqnarray}
 As to $T_{a}u_{4}$, direct computation show
\begin{eqnarray*}
|T_{a}u_{4}(x)|
&\lesssim& M_{2}u(x_{0}),
\end{eqnarray*}
which implies that
\begin{eqnarray}\label{Ec}
\frac{1}{|Q|}\int_{Q}|T_{a}u_{4}(x)|dx\lesssim M_{2}u(x_{0}).
\end{eqnarray}
From (\ref{Eb}) and (\ref{Ec}), it follows that
\begin{eqnarray*}
\frac{1}{|Q|}\int_{Q}|T_{a}u(x)|dx\lesssim M_{2}u(x_{0}).
\end{eqnarray*}
Combining all these estimates, we could get the desired conclusion.

\end{proof}

\section{The proof of Theorem \ref{weight}}
In this section, Theorem \ref{weight} will be proved and the method is inspired by Chen \cite{Chen}. The outline of this proof is that the case $r=2$ and $r=1$ is proved firstly and then the case $1\leq r\leq2$ can be got by some interpolations. Notice that the case $r=2$ is a direct conclusion of Theorem \ref{T1}, that is,

\begin{prop}\label{P1}
Let $0\leq\varrho\leq1,$ $0\leq\delta<1$ and $m=-\frac{n}{2}(1-\varrho).$ Suppose $a(x,\xi)\in S^{m}_{\varrho,\delta}$, $\forall \omega\in A_{p/2}$ with $p>2$. Then there is a constant $C$ independent of $a$ and $u$, such that
$$\|T_{a}u\|_{L^{p}_{\omega}}\leq C\|u\|_{L^{p}_{\omega}}.$$
\end{prop}

Next the case $r=1$ is considered.
\begin{prop}\label{P2}
Let $0\leq\varrho\leq1,$ $0\leq\delta<1$ and $m=-n(1-\varrho).$ Suppose $a(x,\xi)\in S^{m}_{\varrho,\delta}$, $\omega\in A_{p}$ with $1<p<\infty$. Then there is a constant $C$ independent of $a$ and $u$, such that
$$\|T_{a}u\|_{L^{p}_{\omega}}\leq C\|u\|_{L^{p}_{\omega}}.$$
\end{prop}

\begin{proof}
To do this, take advantage of decomposition \eqref{de} and write $T_{j}$ as the following form
\begin{eqnarray*}
T_{j}u(x)
=\int_{\mathbb{R}^n}\int_{\mathbb{R}^n} e^{ i \langle 2^{j\varrho}x-y,\xi\rangle}a(x,2^{j\varrho}\xi) \psi(2^{-j(1-\varrho)}\xi) d\xi u(2^{-j\varrho}y)dy.
\end{eqnarray*}
Denote $a_{j}(x,\xi)=a(2^{-j\varrho}x,2^{j\varrho}\xi) \psi(2^{-j(1-\varrho)}\xi)$, $\tau_{j}u(x)=u(2^{j\varrho}x)$ and
\begin{eqnarray*}
\tilde{T}_{j}u(x)
=\int_{\mathbb{R}^n}\tilde{K}_{j}(x,x-y)u(y)dy
\end{eqnarray*}
with
$$\tilde{K}_{j}(x,w)=\int_{\mathbb{R}^n} e^{ i  \langle w,\xi\rangle}a_{j}(x,\xi) d\xi.$$
Then
$$T_{j}u(x)=\tau_{j}\tilde{T}_{j}\tau_{-j}u(x).$$

Notice that $supp_{\xi} a_{j}(x,\xi)\subset \{\xi:C^{-1}2^{j(1-\varrho)}\leq|\xi|\leq C2^{j(1-\varrho)}\}$
\begin{equation}\label{E7}
|\partial^{\alpha}_{x}\partial^{\beta}_{\xi}a_{j}(x,\xi)|\leq C_{\alpha,\beta} 2^{-jn(1-\varrho)+|\alpha|(\delta-\varrho)}
\end{equation}
for any multi-indices $\alpha,\beta,$ which implies two facts. One is that
\begin{equation}\label{E9}
\|\tilde{T}_{j}u\|_{L^{2}}\lesssim2^{-jn(1-\varrho-\frac{\max\{\delta-\varrho,0\}}{2})}\|u\|_{L^{2}}.
\end{equation}
The other is that for any $\omega\in A_{2}$
\begin{equation}\label{E8}
\|\tilde{T}_{j}u\|_{L^{2}_{\omega}}\lesssim\|u\|_{L^{2}_{\omega}}.
\end{equation}
In fact, \eqref{E9} is a direct conclusion of $L^{2}$ boundedness of pseudo-differential since that $a_{j}(x,\xi)\in S^{-\frac{n\max\{\delta-\varrho,0\}}{2(1-\varrho)}}_{0,\frac{\delta-\varrho}{1-\varrho}}$ with bounds $\lesssim 2^{-jn(1-\varrho-\frac{\max\{\delta-\varrho,0\}}{2})}$ follows from  \eqref{E7}.  For  \eqref{E8}, integrating by parts and \eqref{E7} give that
$$|\tilde{K}_{j}(x,x-y)|\leq C(1+|x-y|^{2})^{-2N}$$
for some positive integer $N$ large enough. So, \eqref{E8} follows immediately from the following pointwise estimate
\begin{eqnarray*}
|\tilde{T}_{j}u(x)|
&\lesssim& \int_{\mathbb{R}^n} (1+|x-y|^{2})^{-2N}|u(y)|dy
\lesssim Mu(x).
\end{eqnarray*}

If $\omega\in A_{2}$, then that $\omega^{1+\epsilon}\in A_{2}$ for some $\epsilon>0$ follows from the reverse H\"{o}lder inequality. Thus, \eqref{E8} gives that
\begin{equation}\label{E81}
\|\tilde{T}_{j}u\|_{L^{2}_{\omega^{1+\epsilon}}}\leq C\|u\|_{L^{2}_{\omega^{1+\epsilon}}}.
\end{equation}
Note that $\omega=(\omega^{1+\epsilon})^{\frac{1}{1+\epsilon}}\times1^{1-\frac{1}{1+\epsilon}}$. For \eqref{E81} and \eqref{E9}, the interpolation theorem with change of measures ( Stein-Weiss \cite{Weiss}) implies that
\begin{equation*}
\|\tilde{T}_{j}u\|_{L^{2}_{\omega}}\leq C2^{-jn(1-\varrho-\frac{\max\{\delta-\varrho,0\}}{2})(\frac{\epsilon}{1+\epsilon})}\|u\|_{L^{2}_{\omega}}.
\end{equation*}
Moreover, the fact that if $\omega\in A_{p}$, then both $\tau_{j}\omega$ and$\tau_{-j}\omega$ belong to $\in A_{p}$ and the $A_{p}$ constan $[\omega]_{p}$ of them is the same as that of $\omega$, gives
\begin{equation*}
\|T_{j}u\|_{L^{2}_{\omega}}=\|\tau_{j}\tilde{T}_{j}\tau_{-j}u\|_{L^{2}_{\omega}}\leq C2^{-jn(1-\varrho-\frac{\max\{\delta-\varrho,0\}}{2})(\frac{\epsilon}{1+\epsilon})}\|u\|_{L^{2}_{\omega}}.
\end{equation*}
Therefore, the following estimate is valid from that $1-\varrho-\frac{\max\{\delta-\varrho,0\}}{2}>0$ and $\epsilon>0$.
\begin{equation*}
\|T_{a}u\|_{L^{2}_{\omega}}\leq C\|u\|_{L^{2}_{\omega}},\quad \forall \omega\in A_{2}.
\end{equation*}
By extrapolation theorem of Rubio de Francia \cite{Francia}, we can get the desired estimate immediately.

\end{proof}

In order to finish the proof of Theorem \ref{weight}, the following $L^{p}$ estimate of pseudo-differential operators is necessary.
\begin{prop}\label{Tz}
Let $0\leq\varrho\leq1,$ $0\leq\delta<1$ and $a(x,\xi)\in S^{m}_{\varrho,\delta}$ with
$$m\leq -n(1-\varrho)|\frac{1}{2}-\frac{1}{p}|-n\frac{\max\{\delta-\varrho,0\}}{\max\{p,2\}}.$$
Then there is a constant $C$ independent of $a$ and $u$, such that
$$\|T_{a}u\|_{L^{p}}\leq C\|u\|_{L^{p}}.$$
\end{prop}
\begin{proof}
For $1<p\leq 2$, it has been proved by \'{A}lvarez \cite{Hounie} and Hounie\cite{Hounie1}. For $2<p<\infty$, it can be proved by interpolation theorem due to Fefferman and Stein \cite{Fefferman1} between Lemma \ref{L2} and the following corollary of Theorem \ref{T1}.
\end{proof}

\begin{corr}\label{BMO}
Let $0\leq\varrho\leq1,$ $0\leq\delta<1$ and $m=-\frac{n}{2}(1-\varrho).$ Suppose $a(x,\xi)\in S^{m}_{\varrho,\delta}$. Then there is a constant $C$ independent of $a$ and $u$, such that
$$\|T_{a}u\|_{{\rm BMO}}\leq C\|u\|_{L^{\infty}}.$$
%$$\|T_{a}u\|_{L^{1}(\mathbb{R}^{n})}\leq C\|u\|_{H^{1}(\mathbb{R}^{n})}$$
\end{corr}
This corollary follows from the fact that $u^{\sharp}\in L^{\infty}$ if and only if $u\in {\rm BMO},$ where ${\rm BMO}$ denotes the spaces of functions of bounded mean oscillation.
\begin{rem}
For $0\leq\varrho\leq1$ and $0\leq\delta<1$ the $(L^{\infty},{\rm BMO})$ estimate in Corollary \ref{BMO} is sharp. In fact, if we consider the symbol $a(\xi)=\phi(\xi)|\xi|^{m}e^{i|\xi|^{1-\varrho}}\in S^{m}_{\varrho,\delta}$, where $\phi$ is a smooth function which vanishes in a neighborhood of origin and equals 1 outside a compact set, then the operator associated to this symbol dose not map $L^{\infty}$ to ${\rm BMO}$ if $m>-\frac{n}{2}(1-\varrho)$(see Miyachi \cite{Miyachi})
\end{rem}

\begin{proof}[Proof of Theorem \ref{weight}]
For $r<p<\infty$, denote $p_{0}=\frac{2p}{r}$ and $p_{1}=\frac{p}{r}$, then $p_{1}>1$, $p_{0}>2$ and $p_{1}\leq p\leq p_{0}$.
By Proposition \ref{P1}, $\forall\omega\in A_{p/r}=A_{p_{0}/2}$, we have
$$\|T_{j}u\|_{L^{p_{0}}_{\omega}}\leq C2^{-jn(1-\varrho)(\frac{1}{r}-\frac{1}{2})}\|u\|_{L^{p_{0}}_{\omega}}.$$
By Proposition \ref{P2}, $\forall\omega\in A_{p/r}=A_{p_{1}}$, we have
$$\|T_{j}u\|_{L^{p_{1}}_{\omega}}\leq C2^{-jn(1-\varrho)(\frac{1}{r}-1)}\|u\|_{L^{p_{1}}_{\omega}}.$$
The Marcinkiewicz interpolation theorem gives
$$\|T_{j}u\|_{L^{p}_{\omega}}\leq C\|u\|_{L^{p}_{\omega}}.$$
Denote $\nu=n(1-\varrho)|\frac{1}{2}-\frac{1}{p}|+n\frac{\max\{\delta-\varrho,0\}}{\max\{p,2\}}$. Then Theorem \ref{Tz} implies that
$$\|T_{j}u\|_{L^{p}}\leq C2^{-j(\frac{n(1-\varrho)}{r}-\nu)}\|u\|_{L^{p}}.$$
By the reverse H\"{o}lder inequality and the interpolation theorem with change of measures again, we get
$$\|T_{j}u\|_{L^{p}_{\omega}}\leq C2^{-j(\frac{n(1-\varrho)}{r}-\nu)(\frac{\epsilon}{1+\epsilon})}\|u\|_{L^{p}_{\omega}}.$$
Note that $\frac{n(1-\varrho)}{r}-\nu>0$ and $\epsilon>0$. The desired estimate can be got immediately.
\end{proof}

$\mathbf{Acknowledgements}$ The author would like to express gratitude to Prof. Wenyi Chen and Prof. Xiangrong Zhu for valuable suggestions concerning this paper.

%%%%%%%%%%%%%%%%%%%%%%%%%%%%%%%%%%%%%%%%%%%%%

\bibliographystyle{Plain}

\end{document}